\documentclass[12pt]{amsart}
\usepackage[left=1in,right=1in,top=1.2in,bottom=1.2in]{geometry}
\usepackage{amsfonts}
\usepackage{amsmath}
\usepackage{amssymb}
\usepackage{amscd}
\usepackage{mathtools}
\usepackage{color}
\newcommand{\Mod}[1]{\ (\mathrm{mod}\ #1)}
\allowdisplaybreaks[1]
\newtheorem{theorem}{Theorem}[section]
\newtheorem{lemma}[theorem]{Lemma}

\newtheorem{corollary}[theorem]{Corollary}
\theoremstyle{definition}

\theoremstyle{remark}

\numberwithin{equation}{section}
\usepackage{hyperref}

\hypersetup{
	bookmarks=true,         		pdftitle={},        pdfauthor={Chiranjit Ray},         colorlinks=true,
	linkcolor={blue},urlcolor={blue},
	citecolor={red}, anchorcolor = {blue}}
\begin{document}
\title[Distribution of  Andrews' Singular Overpartitions $\overline{C}_{p,1}(n)$ ]{Distribution of  Andrews' Singular Overpartitions $\overline{C}_{p,1}(n)$}
\author{Chiranjit Ray}
\address{Mathematics and Data Science Group, Indian Institute of Information Technology Sri City,
Tirupati District - 517 646, Andhra Pradesh, India}
\curraddr{}
\email{chiranjitray.m@iiits.in}

\date{December 2, 2021}
\thanks{} 
\begin{abstract}
Andrews introduced the partition function $\overline{C}_{k, i}(n)$, called singular overpartition, which counts the number of overpartitions of $n$ in which no part is divisible by $k$ and only parts $\equiv \pm i\pmod{k}$ may be overlined. We study the  parity and  distribution results for $\overline{C}_{k,i}(n),$ where $k>3$ and $1\leq i \leq \left\lfloor\frac{k}{2}\right\rfloor$. More particularly,  we prove that for each integer $\ell\geq 2$ depending on $k$ and $i$, the interval $\left[\ell, \frac{\ell(3\ell+1)}{2}\right]$ $\Big($resp.\  $\left[2\ell-1, \frac{\ell(3\ell-1)}{2}\right]
\Big)$ contains an integer $n$ such that $\overline{C}_{k,i}(n)$ is even (resp.\ odd). Finally we study the distribution for $\overline{C}_{p,1}(n)$ where $p\geq 5$ be a prime number.
\end{abstract}
\maketitle
\section{Introduction and statement of results}
A partition $\lambda$ of a positive integer $n$ is a non-increasing sequence of positive integers
$\lambda_1, \lambda_2, \dots, \lambda_k$ whose sum is $n$. Each $\lambda_i$ is called a part of the partition $\lambda$.
In \cite{corteel2004}, Corteel and Lovejoy introduced overpartitions. An overpartition of $n$ is a non-increasing sequence of natural numbers whose sum
is $n$ in which the first occurrence of a number may be overlined.
Andrews~\cite{andrews2015}  defined a new type of  overpartitions,  called singular overpartition to give the overpartition analogous to Rogers--Ramanujan type theorems for ordinary partitions with restricted successive ranks. 
The singular overpartition, $\overline{C}_{k,i}(n)$, 
counts the number of overpartitions of $n$ in which no part is divisible by $k$ and only parts $\equiv \pm i\pmod{k}$ may be overlined. For example, $\overline{C}_{3, 1}(4)=10$ with the relevant partitions being $4, \overline{4}, 2+2, \overline{2}+2, 2+1+1, \overline{2}+1+1, 2+\overline{1}+1, \overline{2}+\overline{1}+1, 1+1+1+1, \overline{1}+1+1+1$. 
For $k\geq 3$ and $1\leq i \leq \left\lfloor\frac{k}{2}\right\rfloor$, the generating function of $\overline{C}_{k,i}(n)$ was derived as follows:
\begin{align*}
	\sum_{n=0}^{\infty}\overline{C}_{k, i}(n)q^n &=\frac{(q^k; q^k)_{\infty}(-q^i; q^k)_{\infty}(-q^{k-i}; q^k)_{\infty}}{(q; q)_{\infty}},
\end{align*}
where $\displaystyle (a; q)_{\infty}:= \prod_{j=0}^{\infty}(1-aq^j)$. The generating functions for $\overline{C}_{3k, k}(n)$, $\overline{C}_{4k, k}(n)$ and $\overline{C}_{6k, k}(n)$ can be written nicely so that we can use various $q$-series identities and theory of modular forms to study the arithmetic and density properties of these three types of singular overpartitions.
For example, the generating function for $\overline{C}_{6k, k}(n)$ can be written as:

\begin{align*}
	\sum_{n=0}^{\infty}\overline{C}_{6k, k}(n)q^n &=\frac{(q^{6k}; q^{6k})_{\infty}(-q^{k}; q^{6k})_{\infty}(-q^{5k}; q^{6k})_{\infty}}{(q; q)_{\infty}}\\
	&=\frac{(q^{6k}; q^{6k})_{\infty}}{(q; q)_{\infty}} \left(\frac{(-q^{k}; q^{6k})_{\infty}(-q^{5k}; q^{6k})_{\infty}(-q^{3k}; q^{6k})_{\infty}}{(-q^{3k}; q^{6k})_{\infty}}\right)\\
	&=\frac{(q^{6k}; q^{6k})_{\infty}}{(q; q)_{\infty}} \left(\frac{(-q^{k}; q^{2k})_{\infty}}{(-q^{3k}; q^{6k})_{\infty}}\right)\\
	&=\frac{(q^{6k}; q^{6k})_{\infty}}{(q; q)_{\infty}} \left(\frac{(-q^{k}; q^{2k})_{\infty}(-q^{2k}; q^{2k})_{\infty}}{(-q^{3k}; q^{6k})_{\infty}(-q^{6k}; q^{6k})_{\infty}} \times \frac{(-q^{6k}; q^{6k})_{\infty}}{(-q^{2k}; q^{2k})_{\infty}}\right)\\
	&=\frac{(q^{6k}; q^{6k})_{\infty}}{(q; q)_{\infty}} \left(\frac{(-q^{k}; q^{k})_{\infty}(-q^{6k}; q^{6k})_{\infty}}{(-q^{2k}; q^{2k})_{\infty}(-q^{3k}; q^{3k})_{\infty}}\right)\\ 
	&=\frac{(q^{6k}; q^{6k})_{\infty}}{(q; q)_{\infty} } \left( \frac{(q^{2k}; q^{2k})_{\infty}(q^{12k}; q^{12k})_{\infty}}{(q^{4k}; q^{4k})_{\infty}(q^{6k}; q^{6k})_{\infty}} \times \frac{(q^{2k}; q^{2k})_{\infty}(q^{3k}; q^{3k})_{\infty}}{(q^{k}; q^{k})_{\infty}(q^{6k}; q^{6k})_{\infty}}\right)\\ 
	&=\frac{(q^{2k}; q^{2k})_{\infty}^2(q^{3k}; q^{3k})_{\infty}(q^{12k}; q^{12k})_{\infty}}{(q; q)_{\infty}(q^{k}; q^{k})_{\infty}(q^{4k}; q^{4k})_{\infty}(q^{6k}; q^{6k})_{\infty}}.\\
\end{align*}
Similarly, for $\overline{C}_{4k, k}(n)$  and $\overline{C}_{3k, k}(n)$  we have respectively
\begin{align*}
(-q^{k}; q^{4k})_{\infty}(-q^{3k}; q^{4k})_{\infty} &=\frac{(q^{2k}; q^{2k})_{\infty}^2}{(q^{k}; q^{k})_{\infty}(q^{4k}; q^{4k})_{\infty}}
\end{align*}
\text{and}
\begin{align*}
(-q^{k}; q^{3k})_{\infty}(-q^{2k}; q^{3k})_{\infty} &=\frac{(q^{2k}; q^{2k})_{\infty}(q^{3k}; q^{3k})_{\infty}}{(q^{k}; q^{k})_{\infty}(q^{6k}; q^{6k})_{\infty}}.
\end{align*}

The singular overpartitions $\overline{C}_{3k, k}(n)$, $\overline{C}_{4k, k}(n)$ and $\overline{C}_{6k, k}(n)$  for some positive integer $k$ are well studied in the literature. Chen, Hirschhorn and Sellers \cite{chen2015} showed that $\overline{C}_{3, 1}(n)$ and $\overline{C}_{4, 1}(2n+1)$ are
 always even. They also proved that $\overline{C}_{6, 2}(n)$ is even (or odd) if and only if $n$ is not (or is) a pentagonal number. In \cite{ahmed2015}, Ahmed and Baruah found congruences modulo $4$, $18$ and $36$ for $\overline{C}_{3, 1}(n)$, infinite families of congruences modulo $2$ and $4$ for $\overline{C}_{8, 2}(n)$, congruences modulo $2$ and $3$ for $\overline{C}_{12, 2}(n)$ and $\overline{C}_{12, 4}(n)$ and congruences modulo $2$ for $\overline{C}_{24, 8}(n)$ and $\overline{C}_{48, 16}(n)$. Naika and Gireesh \cite{naika2016} proved congruences for $\overline{C}_{3, 1}(n)$ modulo $6$, $12$, $16$, $18$, and $24$. They also found infinite families of congruences for $\overline{C}_{3, 1}(n)$ modulo $12$, $18$, $48$, and $72$. In \cite{barman2018}, Barman and the author affirm a conjecture of Naika and Gireesh by proving that $\overline{C}_{3,1}(12n+11)\equiv 0\pmod{144}$ for all $n\geq 0$. Later Barman and the author \cite{rupam2019}  studied the distribution for $\overline{C}_{3,1}(n)$ and proved that $\overline{C}_{3,1}(n)$ is almost always divisible by $2^k$ and $3^k$ for any positive integer $k$. This result was recently studied for more general class of singular overpartition $\overline{C}_{3k, k}(n)$ for some $k$ by Singh and Barman~\cite{singh2021.2, singh2021}.
 
Note that the arithmetic and density properties of singular overpartitions $\overline{C}_{k, i}(n)$ apart from the above three types are not considered in literature well. In this article, we study the parity and distribution results for  $\overline{C}_{p,1}(n),$  where $p\geq 5$ be a prime number. More precisely, we obtain the following results.  
%%%%%%%%%%%%%%%%%%%%%%%%

 \begin{theorem}\label{theorem222.coro1}
If  $\ell\geq 2$ is a positive integer such that $\ell\equiv1 \Mod 3$ and  $p\geq 5$ is a prime, then $\ell(3\ell+1)$ is not of the form $\big(pm^2 \pm  m(p-2)\big)$  for any positive integer $m$. Furthermore, there are infinitely many integers $n$ for which $\overline{C}_{p, 1}(n)$ is an even integer.
 \end{theorem}
%  \noindent For example, if we  consider $\ell=3$, then we have $6, 10 \in [3, 15]$ such that both  $\overline{C}_{5, 1}(6)$ and $\overline{C}_{5, 1}(10)$ are even.
 \begin{corollary}\label{theorem222.coro2}
 For every positive integer $n$ and prime number $p\geq 5$, we have
 	\begin{align*}
 	\left\{1\leq n \leq X: \overline{C}_{p, 1}(n)~ \text{is an even integer} \right\} \geq \alpha \log \log X,
 	\end{align*}
 	where $\alpha>0$ is a constant.
 \end{corollary}

  \begin{theorem}\label{theorem111.coro1}
If  $\ell\geq 2$ is a positive integer such that $\ell\equiv2 \Mod 3$ and  $p\geq 5$ is a prime, then $\ell(3\ell-1)$ is not of the form $\big(pm^2 \pm  m(p-2)\big)$  for any positive integer $m$. Furthermore, there are infinitely many integers $n$ for which $\overline{C}_{p, 1}(n)$ is an odd integer.
 \end{theorem}
%   \noindent When $\ell=3$, we have $22, 34 \in [11, 70]$ such that both $\overline{C}_{7, 1}(22)$ and $\overline{C}_{7, 1}(34)$ are odd. 
 \begin{corollary}\label{theorem111.coro2}
 For every positive integer $n$ and prime number $p\geq 5$, we have
 	\begin{align*}
 	\left\{1\leq n \leq X: \overline{C}_{p, 1}(n) ~ \text{is an odd integer} \right\} \geq \beta \log \log X,
 	\end{align*}
 	where $\beta>0$ is a constant.
 \end{corollary}
We use  $Mathematica$ \cite{mathematica} for our computations.

%%%%%%%%%%%%%%%%%%%%%%%%%%%%%%%%%%%%

\section{Proof of Theorem~\ref{theorem222.coro1} and Theorem~\ref{theorem111.coro1} } 
 To prove our main results %Theorem ~\ref{theorem222}--\ref{theorem111}  
 we need the following lemmas. 
 \begin{lemma}\label{lemma1}
 	For any positive integer $n$, we have 
 	\begin{align*}
 &\sum_{s=0}^{\infty}\overline{C}_{k, i}\left(n-\frac{s(3s-1)}{2}\right) +\sum_{s=1}^{\infty}\overline{C}_{k, i}\left(n-\frac{s(3s+1)}{2}\right)\\
 &\equiv \begin{cases}
      1\Mod{2} & \text{if $n=\frac{1}{2}\big(km^2\pm m(k-2i)\big)$ ~~~ for some $m\in \mathbb{N}$},\\
      0\Mod{2} & \text{otherwise}.
\end{cases}
\end{align*}
 \end{lemma}
 	\begin{proof}
    Ramanujan's general theta function $f(a, b)$ is defined as
 	\begin{align*}
 	f(a, b)=\sum_{n=-\infty}^{\infty}a^{\frac{n(n+1)}{2}}b^{\frac{n(n-1)}{2}} ~~~\text{for}~~~ |ab|<1.
 	\end{align*}
 	The Jacobi triple product identity \cite[Entry 19, p. 36]{Berndt1991} takes the shape 
 	\begin{align*}
 	\index{$f(a,b)$}	f(a,b)=(-a;ab)_{\infty}(-b;ab)_{\infty}(ab;ab)_{\infty}.
 	\end{align*}
One of the most important special cases of $f(a,b)$ is
 	\begin{align*}
f(-q, -q^2)=(q;q)_{\infty}=\sum_{n=-\infty}^{\infty}(-1)^nq^{\frac{n(3n-1)}{2}}.
 	\end{align*}
 	Therefore we have
 	\begin{align}\label{2.1}
 	(q;q)_{\infty}\equiv \sum_{n=0}^{\infty}q^{\frac{n(3n-1)}{2}}+\sum_{n=1}^{\infty}q^{\frac{n(3n+1)}{2}} \Mod {2}.
 	\end{align}
The following identity is due to Andrews~\cite{andrews2015}	
 	\begin{align*}
	\sum_{n=0}^{\infty}\overline{C}_{k, i}(n)q^n 	&=\frac{1}{(q; q)_{\infty}}\left( \sum_{n=-\infty}^{\infty}q^{\frac{k(n^2-n)}{2}+in} \right)\\
	&=\frac{1}{(q; q)_{\infty}}\left(\displaystyle\sum_{n=0}^{\infty}q^{\frac{k(n^2-n)}{2}+in}+\displaystyle\sum_{n=1}^{\infty}q^{\frac{k(n^2+n)}{2}-in}\right).
\end{align*}
 Using \eqref{2.1} along with the above equation we obtain
 	\begin{align*}
 	\sum_{n=0}^{\infty}\overline{C}_{k, i}(n) q^n  \equiv \frac{\displaystyle\sum_{n=0}^{\infty}q^{\frac{k(n^2-n)}{2}+in}+\displaystyle\sum_{n=1}^{\infty}q^{\frac{k(n^2+n)}{2}-in}}{ \displaystyle \sum_{n=0}^{\infty}q^{\frac{n(3n-1)}{2}}+\sum_{n=1}^{\infty}q^{\frac{n(3n+1)}{2}}} \Mod{2}.
 	\end{align*}
 	Therefore,  
 	\begin{align*}
 	&\sum_{n=0}^{\infty}\left(\sum_{k=0}^{\infty}\overline{C}_{k, i}\left(n-\frac{s(3s-1)}{2}\right) +\sum_{k=1}^{\infty}\overline{C}_{k, i}\left(n-\frac{s(3s+1)}{2}\right)\right)q^n\\
 	&\equiv \displaystyle\sum_{n=0}^{\infty}q^{\frac{k(n^2-n)}{2}+in}+\displaystyle\sum_{n=1}^{\infty}q^{\frac{k(n^2+n)}{2}-in} \Mod{2}.
 	\end{align*}
 	Hence, for any positive integer $n$, comparing the coefficients of $q^n$ on both sides of the above congruence, we get the result.
 	\end{proof}
 %%%%%%%%%%%%%%%%%%%%%%%%%%%%%%%%%%%%%%
 \begin{lemma}\label{theorem222} Suppose $k$ and  $\ell\geq 2$ are positive integers. If $\ell(3\ell+1)$ is not of the form $\big(km^2 \pm  m(k-2i)\big)$  for any positive integers $m$ and $i$ such that $1\leq i \leq \left\lfloor\frac{k}{2}\right\rfloor$, then there exists an integer $n\in \left[\ell, \frac{\ell(3\ell+1)}{2}\right]$ such that $\overline{C}_{k, i}(n)$ is an even integer.  
 \end{lemma}
	
\begin{proof} We will prove by contradiction. Let us consider an integer $\ell\geq2$ such that $\ell(3\ell+1)$ is not of the form $\big(km^2 \pm  m(k-2i)\big)$ for any positive integer $m$. Suppose that $\overline{C}_{k, i}(n)$ is odd  when $n\in \left[\ell, \frac{\ell(3\ell+1)}{2}\right]$. We define $\mathcal{G}(k):= \frac{\ell(3\ell+1)}{2}-\frac{k(3k-1)}{2}$  and $\mathcal{H}(k):= \frac{\ell(3\ell+1)}{2}-\frac{k(3k+1)}{2}$ for an integer $k$.  From Lemma~\ref{lemma1}, we have 
	\begin{align}\label{2.4}
	\sum_{k=0}^{\infty}\overline{C}_{k, i}\left(\mathcal{G}(k)\right) +\sum_{k=1}^{\infty}\overline{C}_{k, i}\left(\mathcal{H}(k)\right) \equiv 0 \Mod{2}.
	\end{align} 	
It is easy to check that $\mathcal{G}(k)$ and $\mathcal{H}(k)$ are negative for all $k\geq \ell+1$.  Since $\overline{C}_{k, i}(n)=0$ for all negative integers $n$, it follows  from \eqref{2.4} that 
	\begin{align} \label{new2.2.1}
	&\sum_{k=0}^{\ell}\overline{C}_{k, i}\left(\mathcal{G}(k)\right) +\sum_{k=1}^{\ell}\overline{C}_{k, i}\left(\mathcal{H}(k)\right) \\
\notag	=&\sum_{k=0}^{\ell}\overline{C}_{k, i}\left(\mathcal{G}(k)\right) +\sum_{k=1}^{\ell-1}\overline{C}_{k, i}\left(\mathcal{H}(k)\right) + \overline{C}_{k, i}\left(\mathcal{H}(\ell)\right).
	\end{align}
For any  fixed positive integer $\ell\geq2$, $\mathcal{H}(k)$ is decreasing functions of $k$. Note that  $\mathcal{H}(0)=\frac{\ell(3\ell+1)}{2}$ and
\begin{align*}
    \mathcal{H}(\ell-m)=\frac{1}{2}\big(6\ell m-3m^2+m\big)\geq\frac{1}{2}\big(6\ell -3\ell+4\big)\geq \ell,
\end{align*}
where $m \in \{1, 2, \cdots, \ell-1\}.$
Therefore each $\mathcal{H}(k) \in \left[\ell, \frac{\ell(3\ell+1)}{2}\right]$ for $k \in \{1,2, \dots, \ell-1\}$. In a similar manner, for each $k \in \{0,1,2, \dots, \ell\}$ we can show that each $\mathcal{G}(k) \in \left[\ell, \frac{\ell(3\ell+1)}{2}\right]$. 
By our assumption 
	$\displaystyle\sum_{k=0}^{\ell}\overline{C}_{k, i}\left(\mathcal{G}(k)\right)+\displaystyle\sum_{k=1}^{\ell-1}\overline{C}_{k, i}\left(\mathcal{H}(k)\right)$ is even since  it is a sum of  $(\ell+1)+(\ell-1)=2\ell$ odd numbers. Also we have $\overline{C}_{k, i}\left(\mathcal{H}(\ell)\right)=\overline{C}_{k, i}\left(0\right)=1$. Consequently, the summation  \eqref{new2.2.1} is  odd for $\ell\geq 2$, which is a contradiction to the fact \eqref{2.4}. This concludes our result.
\end{proof}
%%%%%%%%%%%%%%%%%%%%%%%%%%%%%
\begin{proof}[Proof of Theorem~\ref{theorem222.coro1}]
Suppose $\ell\geq 2$ is a positive integer with $\ell \equiv 1\Mod 3$.
For a fixed $\ell$ and a prime number $p\geq 5$, suppose that
$\ell(3\ell+1)= pm^2\pm(p-2)m$ for some $m\in \mathbb{N}$.
Therefore $(p-2)^2+4p\ell(3\ell+1)$ must be a square of an integer.
Then, there exists a positive number $v$ such that
\begin{align}
  \label{coro1.eq1}  v(v+p-2)=p\ell(3\ell+1).
\end{align}
Notice that either $p|v$ or $p|(v+p-2)$. If $p|v$ then $v=ps$ for some $s\in \mathbb{N}$ and from \eqref{coro1.eq1} we have  $ps(ps+p-2)=p\ell(3\ell+1)$. Since $\text{gcd}(p,3)=1$, we obtain 
\begin{align}
  \label{coro1.eq2} s(ps+p-2)&\equiv \ell\Mod3.
\end{align}
The prime $p$ can be considered in two ways like $p\equiv 1 \Mod 3$ or $p\equiv 2 \Mod 3$ and similarly   $s$ can be considered in three ways modulo $3$. 
So there are total six possibility for the pair $(p,s)$ modulo $3$ and it is easy to check that the number $s(ps+p-2)$ is either multiple of $3$ or $s(ps+p-2)\equiv 2\Mod3$. Which contradicts the fact \eqref{coro1.eq2} as $\ell \equiv 1\Mod 3$.\\
Again if $p|(v+p-2)$ then $(v+p-2)=ps$ for some $s\in \mathbb{N}$. Since $\text{gcd}(p,3)=1$, we obtain 
\begin{align*}
   s(ps-p+2)&\equiv \ell\Mod3.
\end{align*}
Then by a similar argument as above we get a contradiction.\\

To prove the second part of the corollary, for $k\in \mathbb{N}$,  we consider 
\begin{align}\label{rec_a_k}
\begin{cases}
 ~& a_0 =\ell,\\ 
  ~& a_{k} = \frac{1}{2}a_{k-1}(3a_{k-1}+1).
\end{cases} 
\end{align}
It is easy to check that for all positive integer $k$, we have  $a_{2k} \equiv 1 \Mod 3$ as $\ell \equiv 1\Mod 3$. Thus, using Lemma~\ref{theorem222}, there exists an integer $n\in [a_{2k}, a_{2k+1}]$ such that $\overline{C}_{p, 1}(n)$ is an even integer. Thus choosing a suitable point from each intervals of   $\bigcup\limits_{m=0}^{\infty}[a_{2m}, a_{2m+1}]
$,  we readily obtain the result.
\end{proof}
\begin{proof}[Proof of Corollary~\ref{theorem222.coro2}] Suppose $n$ is a positive integer. Now we count the number of elements in the set  $$\left\{1\leq n \leq X: \overline{C}_{p, 1}(n)~ \text{is an even integer} \right\}.$$
Next we consider \eqref{rec_a_k} for $a_0=4$ and partition the interval $[1,X]$ as follows
$$[1,X]=[1,4)\cup[a_0,a_2)\cup[a_2,a_4)\cup\dots\cup[a_{2k-2},a_{2k})\cup \dots \cup[a_{\nu},X],$$ 
where $\nu$  is the largest integer such that $a_{\nu}\leq X.$  By Lemma~\ref{theorem222} we can find a positive number $n$ in $\left[a_{2(k-1)}, a_{2k-1}\right]$ such that $\overline{C}_{p, 1}(n)$ is an even integer. Then the number of $n\leq X$  for which $\overline{C}_{p, 1}(n)$ is even  is at
least $\lfloor\nu/2\rfloor$. It remains to find the value of $\nu$ as a function of $X$. For all $k\geq 0$, we have
\begin{align*}
 a_{k} = \frac{a_{k-1}(3a_{k-1}+1)}{2} \leq 2 a_{k-1}^2 \leq 2^{2^{k-1}-1} a_1^{2^{k-1}}\leq 2^{2^k}.
\end{align*}
Since $a_{\nu}\leq X< a_{\nu+1}$,  we see that  $\nu\geq \alpha \log\log X$ for some  constant  $\alpha>0$.
\end{proof}
%%%%%%%%%%%%%%%%%%%%%%%%%

  \begin{lemma}\label{theorem111}
 Suppose $k$ and  $\ell\geq 2$ are positive integers. If $\ell(3\ell-1)$ is not of the form $\big(km^2 \pm  m(k-2i)\big)$  for any positive integer $m$ and $i$ such that $1\leq i \leq \left\lfloor\frac{k}{2}\right\rfloor$, then there exists an integer  $n\in \left[2\ell-1, \frac{\ell(3\ell-1)}{2}\right]$ such that $\overline{C}_{k, i}(n)$ is an odd integer. 
 \end{lemma}
\begin{proof} We will give a proof using the method of contradiction. Consider a  positive integer $\ell\geq 2$ such that $\ell(3\ell-1)$ is not of the form $\big(km^2 \pm  m(k-2i)\big)$  for any positive integer $m$. Assume that $\overline{C}_{k, i}(n)$ is an even number when $n\in \left[2\ell-1, \frac{\ell(3\ell-1)}{2}\right]$. For an integer $k$, let us consider $\mathcal{S}(k):= \frac{\ell(3\ell-1)}{2}-\frac{k(3k-1)}{2}$  and $\mathcal{T}(k):= \frac{\ell(3\ell-1)}{2}-\frac{k(3k+1)}{2}$.  From Lemma~\ref{lemma1} we have 
	\begin{align}\label{2.2}
	\sum_{k=0}^{\infty}\overline{C}_{k, i}\left(\mathcal{S}(k)\right) +\sum_{k=1}^{\infty}\overline{C}_{k, i}\left(\mathcal{T}(k)\right) \equiv 0 \Mod{2}.
	\end{align} 	
It is easy to check that, for all $k\geq \ell+1$ the value of $\mathcal{S}(k)<0$, and  $\mathcal{T}(k)<0$ for all $k\geq\ell$. Since $\overline{C}_{k, i}(n)=0$ for all negative integers $n$, it follows from \eqref{2.2} that
	\begin{align} \label{2.3}
	\notag	&\sum_{k=0}^{\ell}\overline{C}_{k, i}\left(\mathcal{S}(k)\right) +\sum_{k=1}^{\ell-1}\overline{C}_{k, i}\left(\mathcal{T}(k)\right) \\
	\notag=&\sum_{k=0}^{\ell-1}\overline{C}_{k, i}\left(\mathcal{S}(k)\right) +\sum_{k=1}^{\ell-1}\overline{C}_{k, i}\left(\mathcal{T}(k)\right) + \overline{C}_{k, i}\left(\mathcal{S}(\ell)\right) \\
	=&\sum_{k=0}^{\ell-1}\overline{C}_{k, i}\left(\mathcal{S}(k)\right) +\sum_{k=1}^{\ell-1}\overline{C}_{k, i}\left(\mathcal{T}(k)\right) +1.
	\end{align}
Note that, for any  fixed positive integer $\ell\geq2$, $\mathcal{S}(k)$ and $\mathcal{T}(k)$ are decreasing functions of $k$. We check that  $\mathcal{S}(0),~\mathcal{S}(\ell-1), ~\mathcal{T}(1)~\text{and}~\mathcal{T}(\ell-1) \in \left[2\ell-1, \frac{\ell(3\ell-1)}{2}\right].$
Thus, for $\ell\geq2$ and $k\in \{1,2, \dots, \ell-1\}$ the values of  $\mathcal{S}(k),~\mathcal{T}(k)$, and $\mathcal{S}(0)$ lies in the interval $\left[2\ell-1, \frac{\ell(3\ell-1)}{2}\right]$. By our assumption
	$\displaystyle\sum_{k=0}^{\ell-1}\overline{C}_{k, i}\left(\mathcal{S}(k)\right)$ and $\displaystyle\sum_{k=1}^{\ell-1}\overline{C}_{k, i}\left(\mathcal{T}(k)\right) $ are even numbers. Therefore the summation \eqref{2.3} 
	is an odd number, which is a contradiction to \eqref{2.2}. This concludes our result.
\end{proof}
\begin{proof}[Proof of Theorem~\ref{theorem111.coro1}]
Considering $\ell \equiv 2 \Mod{3}$ and proceeding in a similar fashion as in the proof of the first part of Theorem~\ref{theorem222.coro1}, we readily obtain the first part of the Theorem~\ref{theorem111.coro1}.
For the second part of the corollary  we consider 
\begin{align}\label{rec_a_k.2}
\begin{cases}
 ~& a_1 =\ell\equiv 2 \Mod{3} ~ \text{and}\\ 
~& a_{k+1} = \frac{1}{2}a_{k}(3a_{k}-1).
\end{cases} 
\end{align}
It is easy to check that for all positive integer $k$, we have  $a_{k} \equiv 2 \Mod 3$. Thus, using Lemma~ \ref{theorem111}, we can get a suitable point from each intervals of   $\bigcup\limits_{k=0}^{\infty}[2a_{k}-1, a_{k+1}]
$ such that $\overline{C}_{p, 1}(n)$ is an odd integer.
\end{proof}

\begin{proof}[Proof of Corollary ~\ref{theorem111.coro2}] Suppose $n$ is a positive number. Now we count the number of elements of the set  $$\left\{1\leq n \leq X: \overline{C}_{k, i}(n)~ \text{is an odd integer} \right\}.$$
Consider $a_1=2$ and  define $a_{k}$ as in \eqref{rec_a_k.2}. Let us take a partition the interval $[1,X]$ as follows
$$[1,X]=[1,a_2)\cup\dots\cup[a_{k-1},a_{k})\cup \dots \cup[a_{\nu},X],$$ 
where $\nu$  is the largest integer such that $a_{\nu}\leq X.$ From Lemma \ref{theorem111}  we can find a positive number $n$ in $[2a_k-1, a_{k+1}]\subset \left[a_{k}, a_{k+1}\right]$ such that $\overline{C}_{k, i}(n)$ is an odd integer. Then the number of $n\leq X$  for which $\overline{C}_{k, i}(n)$ is odd  is at
least $\lfloor\nu/2\rfloor$. It remains to find the value of $\nu$ as a function of $X$. Now for all $k\geq 0$ we get
\begin{align*}
 a_{k} = \frac{a_{k-1}(3a_{k-1} -1)}{2} \leq \frac{3}{2} a_{k-1}^2 \leq\left( \frac{3}{2}\right)^{2^{k-1}-1}a_1^{2^{k-1}}=\left( \frac{3}{2}\right)^{2^{k-1}-1}2^{2^{k-1}}\leq2^{2^k}
\end{align*}
and hence we get the result immediately.
\end{proof}
\subsection*{Concluding remarks}
Many important partition functions have generating functions which are quotients of theta functions, like the Andrews' singular overpartitions. One can use the  method employed in the article to study the parity results in general for those partition functions. 
For some prime $p\geq5,$
it will also be interesting to see an  arithmetic progression $an+b$ such that $\overline{C}_{p, 1}(an+b)$ is even (resp.\ odd). 
% \subsection*{Acknowledgements} 
% The author has carried out this work at the Indian Institute of Information Technology Sri City (IIITS).

\bibliographystyle{plain}
\bibliography{cray}

\begin{thebibliography}{10}

\bibitem{ahmed2015}
Zakir Ahmed and Nayandeep~Deka Baruah.
\newblock New congruences for {A}ndrews' singular overpartitions.
\newblock {\em Int. J. Number Theory}, 11(7):2247--2264, 2015.

\bibitem{andrews2015}
George~E. Andrews.
\newblock Singular overpartitions.
\newblock {\em Int. J. Number Theory}, 11(5):1523--1533, 2015.

\bibitem{barman2018}
Rupam Barman and Chiranjit Ray.
\newblock Congruences for {$\ell $}-regular overpartitions and {A}ndrews'
  singular overpartitions.
\newblock {\em Ramanujan J.}, 45(2):497--515, 2018.

\bibitem{rupam2019}
Rupam Barman and Chiranjit Ray.
\newblock Divisibility of {A}ndrews' singular overpartitions by powers of 2 and
  3.
\newblock {\em Res. Number Theory}, 5(3):Paper No. 22, 7, 2019.

\bibitem{Berndt1991}
Bruce~C. Berndt.
\newblock {\em Ramanujan's notebooks. {P}art {III}}.
\newblock Springer-Verlag, New York, 1991.

\bibitem{chen2015}
Shi-Chao Chen, Michael~D. Hirschhorn, and James~A. Sellers.
\newblock Arithmetic properties of {A}ndrews' singular overpartitions.
\newblock {\em Int. J. Number Theory}, 11(5):1463--1476, 2015.

\bibitem{corteel2004}
Sylvie Corteel and Jeremy Lovejoy.
\newblock Overpartitions.
\newblock {\em Trans. Amer. Math. Soc.}, 356(4):1623--1635, 2004.

\bibitem{mathematica}
Wolfram~Research{,} Inc.
\newblock Mathematica, {V}ersion 10.0.
\newblock Champaign, IL, 2014.

\bibitem{naika2016}
M.~S. Mahadeva~Naika and D.~S. Gireesh.
\newblock Congruences for {A}ndrews' singular overpartitions.
\newblock {\em J. Number Theory}, 165:109--130, 2016.

\bibitem{singh2021.2}
Ajit Singh and Rupam Barman.
\newblock Certain eta-quotients and arithmetic density of {A}ndrews' singular
  overpartitions.
\newblock {\em J. Number Theory}, 229:487--498, 2021.

\bibitem{singh2021}
Ajit Singh and Rupam Barman.
\newblock New density results and congruences for {A}ndrews' singular
  overpartitions.
\newblock {\em J. Number Theory}, 229:328--341, 2021.

\end{thebibliography}
\end{document}